\documentclass[smallextended]{svjour3}

\usepackage{lipsum} 
\usepackage{amsmath}
\usepackage{amssymb}
\usepackage{dsfont}
\usepackage{multirow}
\usepackage{graphicx}
\usepackage{color}
\usepackage{subfigure}
\usepackage{multirow}
\usepackage{bbm}
\usepackage[sort&compress,numbers]{natbib}	
\usepackage{bibentry}
\nobibliography*

\newcommand{\Exp}{\mathbb{E}}
\newcommand{\Prob}{\mathbb{P}}

\begin{document}

\title{Occupancy times for time-dependent stage-structured models}

\author{George~Chappelle \and Alan~Hastings \and Martin~Rasmussen}

\authorrunning{George~Chappelle, Alan~Hastings, and Martin~Rasmussen}

\date{\today}

\institute{
George Chappelle, Department of Mathematics, Imperial College London, 180 Queen's Gate, London SW7 2AZ, United Kingdom \and
Alan Hastings, Santa Fe Institute, 1399 Hyde Park Road, Santa Fe, NM 87501, USA, and Department of Environmental Science and Policy, University of California, Davis, CA 95616, USA
\and
Martin Rasmussen, Department of Mathematics, Imperial College London, 180 Queen's Gate, London SW7 2AZ, United Kingdom
}

\maketitle

\begin{abstract}
  During their lifetimes, individuals in populations pass through different states, and the notion of an occupancy time describes the amount of time an individual spends in a given set of states. Questions related to this idea were studied in a recent paper by Roth and Caswell for cases where the environmental conditions are constant. However, it is truly important to consider the case where environments are changing randomly or in directional way through time, so the transition probabilities between different states change over time, motivating the use of time-dependent stage-structured models. 
  
  Using absorbing inhomogenous Markov chains and the discrete-time McKendrick--von F{\"o}rster equation, we derive explicit formulas for the occupancy time, its expectation, and its higher-order moments for stage-structured models with time-dependent transition rates. We apply our approach to study a time-dependent model of the Southern Fulmar, and obtain insights into how the number of breeding attempts depends on external conditions that vary through time.
\end{abstract}



\section{Introduction}

The study of structured populations in ecology often is focused on organisms with a life history with distinct yearly events.  A seminal contribution to this work was that of Leslie in the 1940's \cite{Leslie45} who represented the dynamics of age-structured populations using a discrete-time model formulated as a vector with entries representing the number of individuals at different, equally spaced, ages, and a matrix representing yearly transitions with the entries in the top row representing births and the sub-diagonal elements representing the year to year survival probabilities.  Similar representations can be used to study spatially structured populations, and more commonly stage-structured populations \cite{Caswell_01_1,Lefkovitch_65_1}. In the latter case, entries represent yearly transition probabilities between different stages in the population.  For some stage-structured populations, individuals transition back and forth between reproductive and nonreproductive states, and of great interest is the overall time individuals spend in reproductive states during their lifetime.  This behavior typically occurs in organisms where successful reproduction represents a substantial expenditure of energy.  Examples of organisms that do these kind of transitions are  plants, with orchids a prominent example \cite{Shefferson_19}, and large birds such as the Southern Fulmar \cite{Jenouvrier_15_1} and the California condor \cite{Meretsky2000} where individuals transition between breeding and nonbreeding states.

Although ecologists typically think of stage-structured modeling approaches as focusing on the numbers of individuals in different states, a different interpretation is possible if we break up the model into two pieces: one representing reproduction, and the other representing transitions among states not involving reproduction.  This separation of the description of reproduction from the description of transitions between states by organisms (see e.g.~\cite{cushing1998}) is a standard approach that is often useful.  Interpretations of these stage-structured models usually assume that the matrix entries representing these transition rates are constant, ignoring both any effects of density dependence or any influence of environmental changes through time. Consequently, the state vector in the submodel that does not include reproduction could equivalently be considered as the probability that a single organism is at a particular state.  The initial condition for this view is not the number of individuals in different states, but the probability distribution for the states that a focal individual starts in.  Another feature is that organisms are not immortal, so an additional state could be included representing death.  

Under this alternate interpretation, the discrete-time model can be viewed as a finite state inhomogeneous absorbing Markov chain. The different states of the Markov chain correspond to reproductive states or physical location. A new view that emerges from this interpretation is that it makes sense to consider the amount of time (number of years) that an organism spends in some subset of the sates in the model.   For example, an organism may move in and out of breeding states, if breeding one year moves the organism to a nonbreeding state the following year. Observations may easily be able to tell if an organism is breeding or non-breeding, or clearly flowering or nonflowering for plants.  Thus it makes sense to consider how much time an individual spends in the breeding state.  More generally, for a given set of states, the amount of time an individual spends in that set over its lifetime is called the occupancy time.

This intriguing view of dynamics of individuals has been explored theoretically in a recent paper by Roth and Caswell \cite{Roth2018}. The key assumption that the probabilities of transitions between different states are constant though time produces a resulting Markov chain that is homogeneous. Then, using ideas from the theory of Markov chains they develop formulas to calculate various quantities of biological interest including the expected occupancy time of an individual in a target set as well as other aspects. 

Yet, environments, and hence transition probabilities between states for an individual, are not constant through time. On shorter time scales, there is important environmental variability, such as El Ninos and the Pacific decadal oscillation, with similar cycles in other locations around the globe.  Perhaps more importantly, global change is leading to dramatic longer term secular changes in the life histories of individual organisms.  Thus, an important feature to include in demographic models for ecological populations is possible changes in the environment through time, representing either short or long term environmental fluctuations. In this case, the resulting demographic models lead to inhomogeneous Markov chains. In this article, we derive formulas for the distribution and moments of occupancy times for inhomogeneous Markov chains which will highlight the role that environmental variability can play. We will also demonstrate how to calculate and interpret these quantities with a specific application, given by breeder states of the Southern Fulmar bird population.

\section{Occupancy times for inhomogeneous Markov chains}\label{sectheory}

The transition probabilities among states for living organisms at any fixed time can be represented by a substochastic matrix, as no individuals are created, but death is possible. Since we allow for changes in these rates through time, we use a sequence of sub-stochastic matrices through time to represent transition probabilities between states for individuals during their lifetime.  Thus,
let $\{B(n)\}_{n\in\mathbb{N}_0}$  be a sequence of column sub-stochastic matrices in $\mathbb{R}^{d\times d}$, i.e.~for all $n\in\mathbb N_0$, all entries of $B(n)$ are non-negative, and for $j\in\{1,\dots,d\}$, we have $\sum_{i=1}^dB_{ij}(n)\le 1$. The matrix element $B_{ij}(n)$ describes the probability an individual moves from state $j$ to state $i$ at the discrete time step from time $n$ to time $n+1$. Thus, the sum of the entries in the columns are less than one because we allow transition to death, which is not explicitly included as a state.  Accordingly, the probabilities describing the evolution from time $m$ to time $n>m$ are given by the elements of the \emph{transition operator} 
\begin{equation}\label{determinsticevolution}
  \Phi(n,m) := B(n-1)\cdots B(m)\,,
\end{equation}
corresponding to the linear nonautonomous difference equation $x_{n+1}=B(n)x_n$, see  \cite{Kloeden_11_2,Poetzsche_10_2}.

Although the question of occupancy times can be analysed in a purely deterministic nonautonomous setting, we aim at a stochastic description in order to make the analysis of statistical quantities more convenient. Let $\{\bar X_n\}_{n \in\mathbb N_0}$ be the corresponding inhomogeneous absorbing Markov chain on the finite state space $S=\{1,\dots, d\}$ that starts in the probability vector $v\in[0,1]^d$, and note that we have
\begin{displaymath}
  \mathbb{P}\{\bar X_{0}=i\}= v_i \quad\mbox{for all } i\in\{1,\dots, d\}
\end{displaymath}
and
\begin{displaymath}
  \mathbb{P}\{\bar X_n = i \,|\, \bar X_m = j\} =\Phi_{ij}(n,m) \quad\mbox{for all } n>m \mbox{ and } i,j\in\{1,\dots,d\}\,.
\end{displaymath}

We now extend our model to include death.  We define the \emph{absorption probabilities} at time $n\in\mathbb N_0$ by
$b(n) = (\operatorname{Id}-B(n))^T\mathds{1}_d$, where $\mathds{1}_d = (1,\dots,1)^T\in\mathbb R^{d}$. It is convenient to explicitly include an additional state representing death, and hence we can extend our absorbing Markov chain $\{\bar X_n\}_{n\in\mathbb N_0}$ to obtain a full Markov chain $\{X_n\}_{n\in\mathbb N}$, which is generated using the  transition matrices \begin{displaymath}
    C(n) := \begin{pmatrix}
    B(n) & 0 \\
    b(n)^T & 1
    \end{pmatrix} \quad\mbox{for all }n\in\mathbb N_0\,.
\end{displaymath}
We note that the Markov chain $\{X_n\}_{n\in\mathbb N}$ has values in  the extended state space $S\cup\{d+1\}$.
We define the \emph{lifetime} of the Markov chain $\{X_n\}_{n\in\mathbb N_0}$ by
\begin{displaymath}
N := \min\big\{ n \in \mathbb N :  X_{n} = d+1\big\}\,.
\end{displaymath}
Note that the lifetime $N$ describes the random time that an individual enters the absorbing state $d+1$, that is the time until death. Its probability distribution can be calculated explicitly. 

\begin{proposition}
The probability distribution of the absorption time is given by 
\begin{equation*}
\mathbb{P}\{ N = n\} = b(n-1)^T \Phi(n-1,0)v \quad \mbox{for all $n\in\mathbb N$}\,.
\end{equation*}
\end{proposition}

\begin{proof}
For all $n\in\mathbb N$, we have
\begin{align*}
\Prob\{N=n\} &= \sum_{j=1}^d \Prob\{X_{n}=d+1, X_{n-1}=j\}  \\
&=\sum_{j=1}^d\Prob\{ X_{n}=d+1\,|\, X_{n-1}=j\}\Prob\{X_{n-1}=j\} \\
&= \sum_{j=1}^d b_j(n-1) (\Phi(n-1,0)v)_j = b(n-1)^T\Phi(n-1,0)v\,,
\end{align*}
which finishes the proof of this proposition.
\end{proof}

We are now interested in the random amount of time an individual spends in a subset, such as the breeding state, of the state space consisting of all possible individual states, up to a certain time.
\begin{definition}
For $R\subset S$, we define the \emph{$R$-occupancy time} at time $n$ by 
\begin{equation*}
A_{R}(n) = \#\big\{k\in\{0,\dots, n-1\} : X_k \in R \big\}\quad \mbox{for all } n\in\mathbb N\,,
\end{equation*}
where $\#$ denotes the number of elements of the set, and we define $A_R(0)=0$. The \emph{lifetime $R$-occupancy time} 
\begin{displaymath}
  \tau_{ R}=A_{ R}(N)
\end{displaymath}
is  the amount of time a individual spends in $R$ over its lifetime. 
\end{definition}
In order to compute the distribution of the lifetime $R$-occupancy time and its moments, it is useful to consider the joint distribution of the $R$-occupancy time and the state of the Markov chain at time $n$, given by 
\begin{equation}\label{pjdef}
  p^{R}_j(a,n) :=  \mathbb{P}\{ A_{R}(n)=a, X_n=j\} \quad \mbox{for all } j\in S\,,\, n\in\mathbb N_0 \mbox{ and } a\in\mathbb N_0\,.
\end{equation}

We note that this is not a probability distribution since there is a nonzero probability of death. In fact, the average mass available at time $n\in\mathbb N$ is given by $\sum_{j=1}^d \sum_{a=0}^{\infty} p_j^{R}(a,n) = (1,\dots, 1) \Phi(n,0)v$, which is less than $1$ in general.

The following proposition shows that this joint distribution evolves according to a partial difference equation, which is a generalisation of the discretised McKendrick--von F{\"o}rster equation for age dependent population dynamics \cite{McKendrick_26_1}. 

\begin{proposition}\label{prop1}
The time evolution of the distribution $p^{R}_j(a,n)$  is governed by the partial difference equation
\begin{equation}
p^{R}_j(a,n+1) = \sum_{k\not\in R} B_{jk}(n)p_k^{ R}(a,n) + \sum_{k\in R} B_{jk}(n)p_k^{ R}(a-1,n) \,,
\label{eq: dist_ev}
\end{equation}
for any $a\in \mathbb N$ and $n\in \mathbb N_0$,
with initial conditions
\begin{equation*}
p_j^{ R}(a,0) = \left\{
      \begin{array}{c@{\quad:\quad}l}
       v_j & a=0  \\
       0 & a>0 
      \end{array}
      \right. \quad \mbox{for all } j\in S\,,
\end{equation*}
and boundary conditions governed by
\begin{displaymath}
p_j^{ R}(0,n+1) = \sum_{k\not\in R}B_{jk}(n)p_k^R(0,n) \quad \text{ for all } j\in S \mbox{ and } n\in\mathbb N_0\,.
\end{displaymath}
\end{proposition}

\begin{proof}
For the evolution equation, we have for all $a\in\mathbb N$ and $n\in\mathbb N_0$ that
\begin{align*}
    p_j^{ R}(a,n+1) &= \Prob\{A_{ R}(n+1)=a, \, X_{n+1}=j\} \\
    &=\Prob\{ A_{ R}(n+1)=a,\, X_{n+1}=j,\, X_n\in R\} \\
    &\qquad+ \Prob\{ A_{ R}(n+1)=a,\, X_{n+1}=j,\, X_n\not\in R\} \\
    &=\Prob\{ A_{ R}(n)=a-1,\, X_{n+1}=j,\, X_n\in R\} \\
    &\qquad+\Prob\{ A_{ R}(n)=a,\, X_{n+1}=j,\, X_n\not\in R\} \\
    &=\sum_{k\in R}\Prob\{X_{n+1}=j\,| X_n=k\}\Prob\{A_{ R}(n)=a-1,\, X_n=k\}  \\
    &\qquad +\sum_{k\not\in R}\Prob\{X_{n+1}=j\,| X_n=k\}\Prob\{A_{ R}(n)=a,\, X_n=k\} \\
    &= \sum_{k\in R} B_{jk}(n)p_k^{ R}(a-1,n) + \sum_{k\not\in R} B_{jk}(n)p_k^{ R}(a,n)\,,
\end{align*}
which proves \eqref{eq: dist_ev}. The initial conditions are satisfied by definition, and the validity of the boundary conditions follows from 
\begin{align*}
    p_j^{ R}(0,n+1) &= \Prob\{A_{ R}(n+1)=0, \, X_{n+1}=j\}\\
    &=\Prob\{ A_{ R}(n+1)=0,\, X_{n+1}=j,\, X_n\not\in R\} \\
    &=\Prob\{ A_{ R}(n)=0,\, X_{n+1}=j,\, X_n\not\in R\} \\
    &= \sum_{k\not\in R}\Prob\{X_{n+1}=j\,| X_n=k\}\Prob\{A_{ R}(n)=0,\, X_n=k\} \\
    &= \sum_{k\not\in R} B_{jk}(n)p_k^{ R}(0,n)\,,
\end{align*}
for all $j\in S$ and $n\in\mathbb N$.
\end{proof}

We define $p^{ R}(a,n):=(p_1^{ R}(a,n),\dots,p_d^{ R}(a,n))^T$ for all $a,n\in\mathbb N_0$. Note that one can write equation \eqref{eq: dist_ev} more compactly in matrix form
\begin{displaymath}
  p^{R}(a,n+1) = B(n)\big(R p^{ R}(a-1,n)+ (\operatorname{Id}-R)p^{ R}(a,n)\big)\,
\end{displaymath}
where we have allowed $R$ to also stand for the diagonal binary matrix describing the subset $R\subset S$, defined by
\begin{equation*}
     R_{ij} = \left\{
      \begin{array}{c@{\quad:\quad}l}
       1 & i=j\in  R \\
       0 & \text{otherwise}
      \end{array}
      \right. \quad \mbox{for all } i,j\in S\,.
\end{equation*} 

The following theorem explains how to calculate the distribution and moments of the lifetime $R$-occupancy time $\tau_R$.

\begin{theorem}\label{theo1} For any $R\subset S$, the distribution of the lifetime $R$-occupancy time $\tau_R$ is given by 
\begin{equation}
  \Prob\{\tau_{ R}=a\}= \sum_{n=0}^{\infty} b(n)^T \left(Rp^{ R}(a-1,n) + (\operatorname{Id}-R)p^R(a,n)\right)        \label{eq:distribution}
\end{equation}
for all $a\in\mathbb N_0$. The moments of $\tau_{R}$ are given by
\begin{equation}
  \mathbb{E}(\tau_{ R}^k)=\sum_{n=0}^{\infty}b(n)^T\left( m_k^{ R}(n) + R\sum_{j=1}^k {k\choose j}m_{k-j}^{ R}(n)\right) \quad \mbox{for all } k\in\mathbb N\,,
  \label{eq:moments}
\end{equation}
where $m^{R}_k(n) = \sum_{a=1}^{n}a^k p^{ R}(a,n)\in\mathbb{R}^d$ for all $k\in\mathbb N_0$. We demonstrate how to efficiently calculate the terms $m_k^R(n)$ below in Proposition~\ref{prop2}.
\end{theorem}

\begin{proof} For all $a\in\mathbb N$, we have
\begin{align*}
    \Prob\{\tau_{ R}&=a\} = \Prob\{ A_{ R}(N)=a\} = \sum_{n=1}^{\infty} \Prob\{A_{ R}(n)=a, N=n\} \\
    &= \sum_{n=1}^{\infty} \sum_{j=1}^d \Prob\{A_{ R}(n)=a, X_{n-1}=j, X_{n}=d+1\} \\
    &= \sum_{n=1}^{\infty} \sum_{j=1}^d \Prob\{X_{n}=d+1\, | \, X_{n-1}=j\}\Prob\{A_{ R}(n)=a, X_{n-1}=j\} \\ 
    &= \sum_{n=1}^{\infty}\Big(\sum_{j\in R}\Prob\{X_n=d+1\,|\,X_{n-1}=j\}\Prob\{A_R(n-1)=a-1,X_{n-1}=j\} \\
    &\qquad+ \sum_{j\not\in R}\Prob\{X_n=d+1\,|\,X_{n-1}=j\}\Prob\{A_R(n-1)=a, X_{n-1}=j\}\Big) \\
    &=\sum_{n=0}^{\infty}b(n)^T\left(Rp(a-1,n)+(\operatorname{Id}-R)p(a,n)\right).
\end{align*}
Note that from the second to the third line, we used conditional independence of the random variables $X_{n}$ and $A_R(n)$ given $X_{n-1}$. This proves \eqref{eq:distribution}. To prove \eqref{eq:moments} we observe
\begin{align*}
    \mathbb{E}(\tau_R^k)&=\sum_{a=1}^{\infty}a^k\Prob\{\tau_R=a\} \\
    &= \sum_{a=1}^{\infty}a^k\sum_{n=0}^{\infty}b(n)^T(Rp(a-1,n)+(\operatorname{Id}-R)p(a,n)) \\
    &= \sum_{n=0}^{\infty}b(n)^T\sum_{a=0}^{\infty}(a+1)^k(Rp(a,n)+(\operatorname{Id}-R)p(a+1,n)) \\
    &=\sum_{n=0}^{\infty}b(n)^T\left( m_k^{ R}(n) + R\sum_{j=1}^k {k\choose j}m_{k-j}^{ R}(n)\right)\, .
\end{align*}
\end{proof}

We now demonstrate how to calculate the terms $m^{R}_k(n) = \sum_{a=1}^{n}a^k p^{ R}(a,n)$ without having to solve equation \eqref{eq: dist_ev}. Fundamental is the observation that for a fixed $k\in\mathbb N$, the sequence $n\mapsto m^R_k(n)$ solves a difference equation, involving also terms $m^R_\ell(n)$, where $\ell<k$. 

\begin{proposition}\label{prop2}
Consider $R\subset S$. Then the sequence $n\mapsto m_0^R(n)$ satisfies the difference equation
\begin{equation*}
    m_0^R(n+1) = B(n) m_0^R(n)\,,
\end{equation*}
For  $k\in\mathbb N$, the sequence $n\mapsto m_k^{ R}(n)$ satisfies the difference equation
\begin{equation*}
    m_k^{ R}(n+1) = B(n)\left( m_k^{ R}(n) + R\sum_{j=1}^k {k\choose j}m_{k-j}^{ R}(n)\right)\,.
\end{equation*}
The initial conditions are given by $m_0^R(0)=v$ and  $m_k^R(0)=0$ for all $k\in\mathbb N$.
\end{proposition}
\begin{proof}
Firstly,  the definition of $m_0^R$ and \eqref{pjdef} imply that $m_0^R(n) = \Phi(n-1,0)v$, which proves that it satisfies the above difference equation. Next, for $k\in\mathbb N$, we use the matrix form 
\begin{displaymath}
  p^{R}(a,n+1) = B(n)\big(R p^{ R}(a-1,n)+ (\operatorname{Id}-R)p^{ R}(a,n)\big)\,
\end{displaymath}
of \eqref{eq: dist_ev} and obtain
\begin{align*}
    m_k^{R}(n+1)&=\sum_{a=1}^{\infty}a^k p^{R}(a,n+1) \\ &=\sum_{a=1}^{\infty}a^kB(n)\big( R p^{R}(a-1,n)+ (\operatorname{Id}-R)p^{ R}(a,n)\big)\\
    &= B(n)\left(R\sum_{a=0}^{\infty}(a+1)^k p^{R}(a,n)+ (\operatorname{Id}-R)\sum_{a=0}^{\infty}a^k p^{ R}(a,n) \right) \\
    &=B(n)\left( \sum_{a=0}^{\infty}a^k p^{R}(a,n)+  \sum_{j=1}^k{k\choose j}R\sum_{a=0}^{\infty} a^{k-j}p^{R}(a,n) \right) \\ 
   &= B(n)\left( m_k^{ R}(n) + R\sum_{j=1}^k {k\choose j}m_{k-j}^{ R}(n)\right)\,.
\end{align*}
\end{proof}

We close this section by a short discussion how the analysis is simplified when looking at certain special cases.

We first consider the special case of determining the future life of an organism, $R=S$. If we assume that $X_n=j$ for some $n\in\mathbb N_0$ and $j\in S$, then this implies that $A_S(n)=n$, and hence, $p_j^S(n,n)=\mathbb P\{X_n = j\} = \Phi(n,0)v$, and we have $p_j^S(a,n)=0$ whenever $a\not=n$. A consequence of this is that the partial difference equation \eqref{eq: dist_ev} from Proposition~\ref{prop1} reduced to a (non-partial) difference equation of the form 
\begin{displaymath}
  p^S(n+1,n+1) = B(n)p^S(n,n)\quad \mbox{for all } n\in\mathbb N_0\,,
\end{displaymath}
with $p^S(0,0)=v$. Formula \eqref{eq:distribution} in Theorem~\ref{theo1} then simplifies to
\begin{displaymath}
  \mathbb P\{\tau_S = a\} = b(a-1)^Tp^S(a-1,a-1) = b(a-1)\Phi(a-1,0)v\,,
\end{displaymath}
and the formulas in Proposition~\ref{prop2} simplify slightly, since the matrix $R$ is the identity matrix and can be removed.

Now consider the special case that the transition probabilities between stages are constant through time, so the Markov chain is homogeneous, which has been treated in Roth and Caswell \cite{Roth2018} using a different approach. Since the matrix $B(n)$ and the vector $b(n)$ do not depend on $n\in\mathbb N_0$ in this case, the probability distribution of the absorption time is given in this special case by
\begin{displaymath}
\mathbb{P}\{N=n\} = b^TB^{n-1}v\,,
\end{displaymath}
which is known as a phase-type distribution \cite{Neuts_81_1}. The formulas in Theorem~\ref{theo1} and Proposition~\ref{prop2} do not simplify, except that the matrix $B$ and the vector $b$ do not depend on time.

Finally, in case the Markov chain is homogeneous and $R=S$, the above discrete-time phase-type distribution is also the distribution of the occupancy time $\tau_S$, and respective quantities for expectation and higher-order moments can be obtained more easily from \cite{Neuts_81_1}; see also \cite{Metzler_18_1}, where the continuous-time case is treated via absorbing Markov processes.

\section{Occupancy times in breeder stages of the Southern Fulmar}
We  illustrate our theory of occupancy times for inhomogeneous Markov chains using a stage-structured model describing the life cycle of the Southern Fulmar. The Southern Fulmar is a sea bird with colonies along the coast of Antarctica. There are four states that the bird can be in, the structure of the model is shown in Figure \ref{fig:graph}. 

\begin{figure}[h]
    \begin{center}
         \includegraphics[width=.8\textwidth]{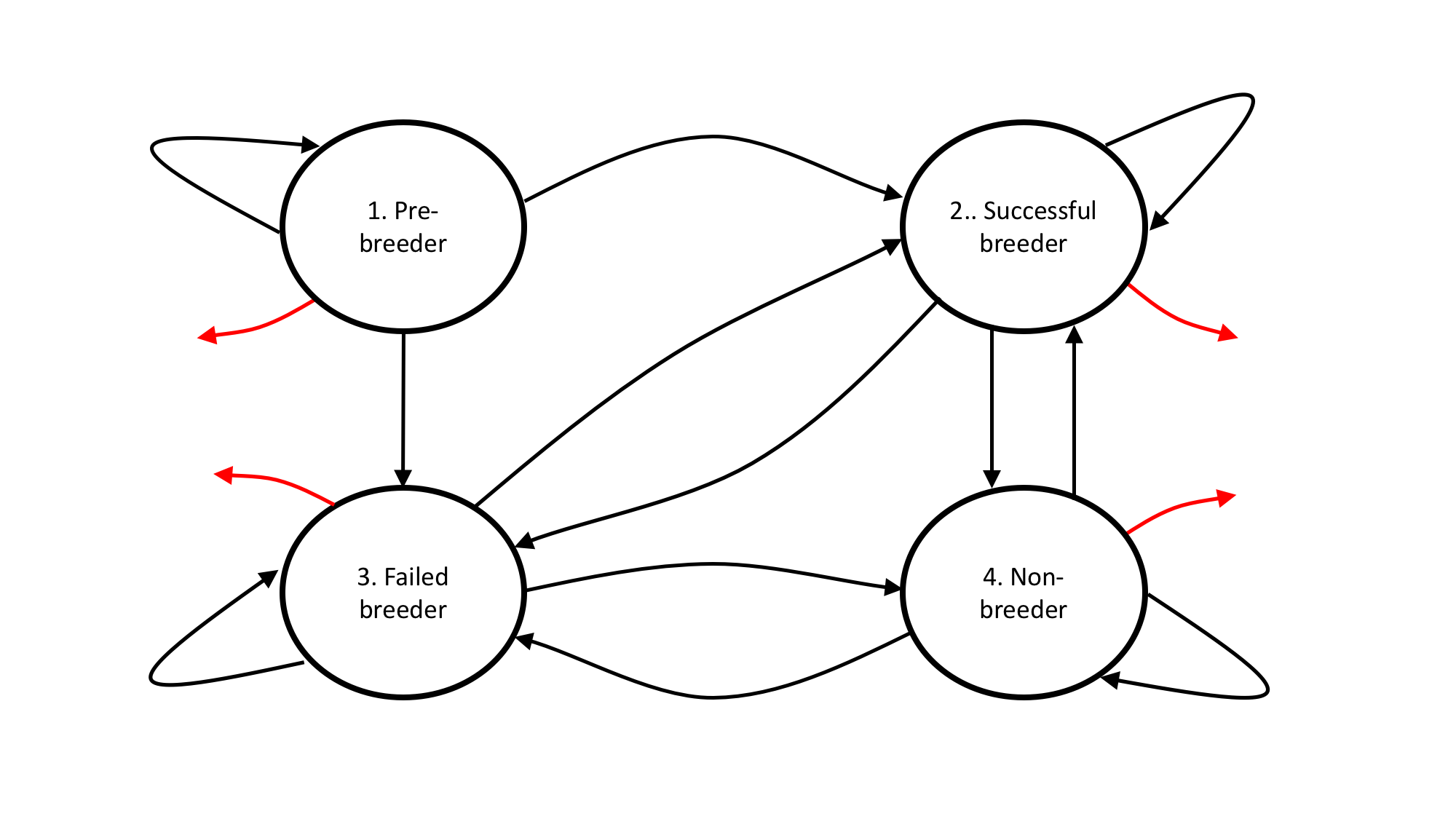}
        \caption{Shows the four states of the Southern Fulmar model and how they are connected. Pre-breeder refers to birds which have yet to bred for the first time. Successful breeders have raised a chick that year. Failed breeders have found a mate but either failed to hatch an egg or the chick died. Finally non-breeders refers to birds who have previously bred but have no mate that year.}
        \label{fig:graph}
    \end{center}
\end{figure}

The rates of transition between these four states depend on the ice conditions in that year. In Jenouvrier,  P{\'e}ron, and Weimerskirch~\cite{Jenouvrier_15_1}, the years are classified as favourable, unfavourable or ordinary. In favourable years the colony is close to the edge of the sea ice. This means foraging trips are shorter and more likely to be successful. 

In Subsection~\ref{subsecdet}, we first consider deterministic, but time-dependent, rates that describe transitions between these four states, and in Subsection~\ref{subsecrand}, we use a random model to describe changes in the external environment.

\subsection{Deterministic external environments}\label{subsecdet}

In \cite{Jenouvrier_15_1}, the transition rates between the four different states have been obtained using satellite data between 1979 and 2010. This involved classification of the three different ice conditions (favourable, ordinary, unfavourable), see Figure~\ref{fig:years}, and three different transition matrices ($U_f$, $U_o$, $U_u$) corresponding to the possible ice conditions:
\begin{displaymath}
U_f = \begin{pmatrix}
0.828 & 0 & 0 & 0 \\
0.06624 & 0.72912 & 0.62244 & 0.40176 \\
0.02576 & 0.18228 & 0.24206 & 0.15624 \\
0 & 0.0186 & 0.0455 & 0.342
\end{pmatrix}\,,
\end{displaymath}
\begin{displaymath}
U_o = \begin{pmatrix}
0.9016 & 0 & 0 & 0 \\
0.011408 & 0.66737 & 0.49312 & 0.1809 \\
0.006992 & 0.18823 & 0.24288 & 0.0891 \\
0 & 0.0744 & 0.184 & 0.63
\end{pmatrix}\,,
\end{displaymath}
\begin{displaymath}
U_u = \begin{pmatrix}
0.9154 & 0 & 0 & 0 \\
0.002392 & 0.4873 & 0.25147 & 0.0468 \\
0.002208 & 0.1895 & 0.23213 & 0.0432 \\
0 & 0.2632 & 0.4464 & 0.81
\end{pmatrix}\,.
\end{displaymath}

\begin{figure}[h]
    \begin{center}
         \includegraphics[width=.8\textwidth]{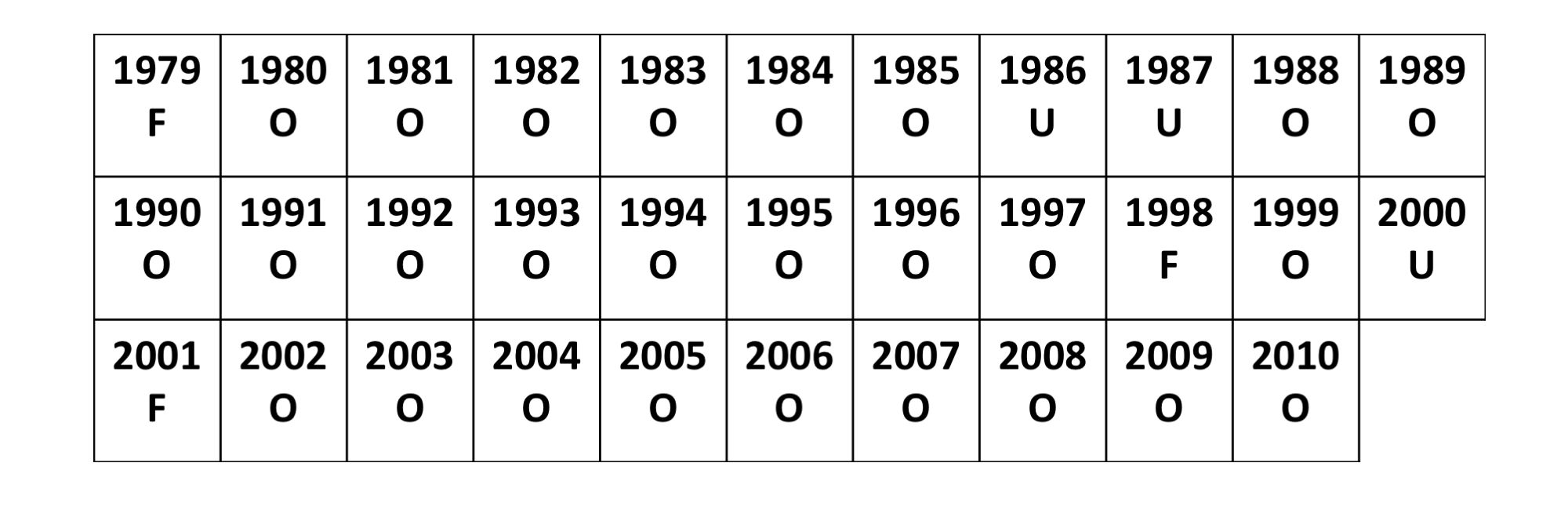}
        \caption{Shows the sequence of ice conditions between 1979 and 2010. The favorability of a year is determined by the distance between the ice edge and the colony, and the total sea ice area.}
        \label{fig:years}
    \end{center}
\end{figure}

We are interested in how many years of a bird's life it attempts to breed. Note that a breeding attempt occurs if the bird is either in the Successful Breeder or Failed Breeder state. This means that, using the notation of the previous section, we are interested in the set  $R=\{2,3\}$ and the lifetime $R$-occupancy time $\tau_R$ is the number of breeding attempts a bird makes in its life. In Figure \ref{fig:tau_R} we use equation \eqref{eq:moments} to calculate the expected value and coefficient of variation for $\tau_R$. In Figure~\ref{fig:tau_R2}, we use equation \eqref{eq:distribution} to calculate the distribution of $\tau_R$ for the four scenarios. 

\begin{figure}[h]
\begin{center} 
\subfigure[]{
\includegraphics[width=0.47\textwidth]{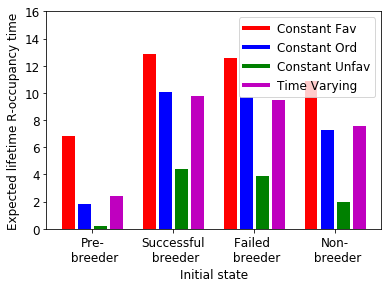}
\label{subfigure:Etau_R}
}
\subfigure[]{
\includegraphics[width=0.47\textwidth]{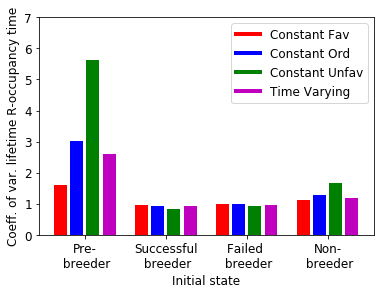}
\label{subfigure:CVtau_R}
}
\end{center}
\caption{Compares the expected valued \subref{subfigure:Etau_R} and the co-efficient of variation \subref{subfigure:CVtau_R} of the lifetime $R$-occupancy time $\tau_R$ for different initial states. We compare the three autonomous scenarios of constant conditions with the time varying conditions described in Figure \ref{fig:years}.}
\label{fig:tau_R}
\end{figure}  

\begin{figure}[h]
\begin{center} 
\subfigure[]{
\includegraphics[width=0.47\textwidth]{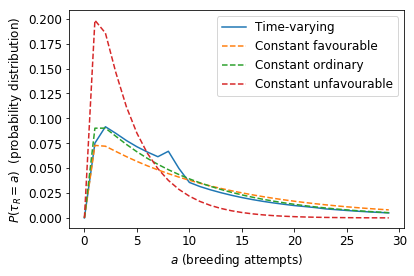}
\label{subfigure:Ptau_R}
}
\subfigure[]{
\includegraphics[width=0.47\textwidth]{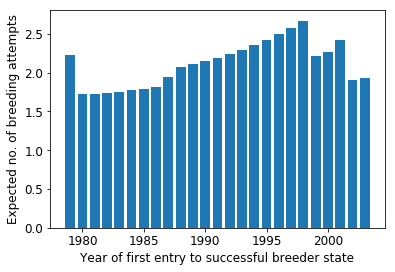}
\label{subfigure:birthdate}
}
\end{center}
\caption{Panel \subref{subfigure:Ptau_R} shows the full probability distribution $a\mapsto\mathbb{P}\{\tau_R=a\}$ of the lifetime $R$-occupancy time for individuals initialised in the successful breeder state. We compare the distribution for the time-dependent scenario shown in Figure \ref{fig:years} with constant conditions. Panel \subref{subfigure:birthdate} shows the expected value of the lifetime $R$-occupancy time as a function of the year they first entered the successful breeder state.}
\label{fig:tau_R2}
\end{figure}  

We conclude that for the time-dependent scenario shown in Figure~\ref{fig:years}, the statistical properties of the lifetime $R$-occupancy time is quite similar to the scenario where all the years are ordinary. The effects of the occasional favourable and unfavourable years balance each other out, but such a conclusion would not have been possible using a purely time-independent analysis of a homogeneous Markov chain.

\subsection{Randomly varying external environments} \label{subsecrand}

The formulas presented in the Section~\ref{sectheory} are valid for an arbitrary sequence of time-dependent transition matrices $B=\{B(n)\}_{n\in\mathbb{Z}}$. A particularly relevant source of such time-dependence arises from a Markov chain in a randomly changing environment. In this subsection, we formulate this mathematically and apply it to the Southern Fulmar example. 

A simple way to generate a random sequence of transition matrices is as follows. Suppose we have three numbers $P_f$, $P_u$, $P_o\in [0,1]$  describing the probability of a favourable, unfavourable and ordinary year respectively. We clearly must have
\begin{equation*}
    P_f+P_u + P_o=1.
\end{equation*}
We can generate a random sequence of transition matrices by picking the matrices $U_f$, $U_u$, $U_o$ according to the probabilities $P_f$, $P_u$, $P_o$ independently at each time-step. We  use $\mu$ to denote the probability measure on the set of sequences of matrices which is described above. For example we have 
\begin{equation}
    \mu\{B(0)=U_o,\, B(1)=U_f,\, B(2)=U_o\} = P_oP_fP_o.
\end{equation}
We note that there are now two separate sources of randomness contributing to lifetime $R$-occupancy time
\begin{enumerate}
    \item The individual level randomness from the Markov chain. Different individuals born at the same time take different demographic paths due to this level of randomness. We use $\mathbb{E}_{\mathbb{P}_B}$ to denote expectations taken over only this level of randomness. Here we emphasise the dependence on the specific sequence of matrices $B$. 
    \item The population level randomness from the varying conditions. All individual in the population at a given time experience the same randomly changing conditions. We use $\mathbb{E}_{\mu}$ to denote expectation taken over only this level of randomness. 
\end{enumerate}

In Figure \ref{fig:rand} we show how the expected value and the co-efficient of variation of the lifetime R-occupancy time depend on the probabilities $P_f$, $P_u$ and $P_o$ of the three different ice conditions. 

\begin{figure}[h]
\begin{center} 
\subfigure[]{
\includegraphics[width=0.47\textwidth]{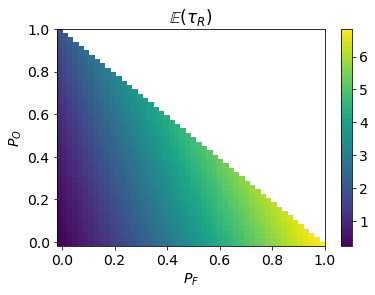}
\label{subfigure:EE}
}
\subfigure[]{
\includegraphics[width=0.47\textwidth]{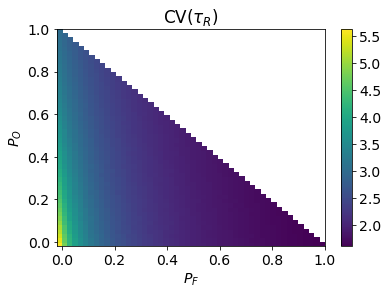}
\label{subfigure:CV}
}
\end{center}
\caption{Shows the expected value and the co-efficient of variation of the lifetime $R$-occupancy time as a function of the probabilities $(P_f, P_o, P_u)$. Here the expectation are taken over both the individual level randomness and the population level. For each $(P_f, P_o, P_u)$ we sample 2000 random sequences of matrices. For each random sequence of transition matrices we use the formulas in the previous section to calculate the path-wise expected value and variance. Then we average over the 2000 samples to approximate the statistics. Individuals are initialised in the pre-breeder state. It is helpful in understanding this figure to note that the three corners of the triangle correspond to constant deterministic conditions as shown in Figure \ref{fig:tau_R}. }
\label{fig:rand}
\end{figure}  

It is interesting to consider how much each level of randomness contributes to the uncertainty expressed by the co-efficient of variation shown Figure \ref{fig:rand}\subref{subfigure:CV}.  
We can decompose the variance to show the contribution of the two sources of randomness and obtain
\begin{align*}
    \mathrm{Var}(\tau_{ R})&=\Exp_{\mu}\Exp_{\Prob_B}\Big[(\tau_{ R}-\Exp_{\mu}\Exp_{\Prob_B}(\tau_{ R}))^2\Big] \\
    &= \Exp_{\mu}\Big[ \Exp_{\Prob_B}(\tau_{ R}^2)-2\Exp_{\Prob_B}(\tau_{ R})\Exp_{\mu}\Exp_{\Prob_B}(\tau_{ R})+\Exp_{\mu}\Exp_{\Prob_B}(\tau_{ R})^2\Big] \\
    &= \Exp_{\mu}\mathrm{Var}_{\Prob_B}(\tau_{ R}) + \mathrm{Var}_{\mu}\Exp_{\Prob_B}(\tau_{ R})\,.
\end{align*}
The total variance is the sum of the expected value of the individual level variance and the variance of the individual level expected value. In Figure~\ref{fig:rand_contribution}, we see that, for our example, most of the variance is explained by the individual level randomness rather than the population level.

\begin{figure}[h]
\begin{center} 
\subfigure[]{
\includegraphics[width=0.47\textwidth]{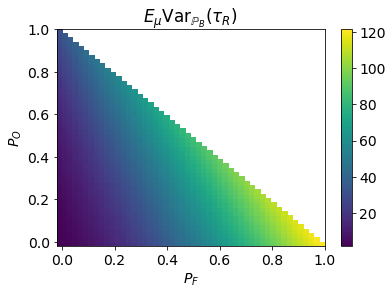}
}
\subfigure[]{
\includegraphics[width=0.47\textwidth]{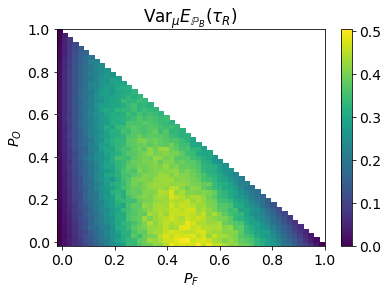}
}
\end{center}
\caption{Shows the contributions of the two sources of randomness in the lifetime $R$-occupancy time. The variance in the lifetime-$R$ occupancy times which comes from the randomly changing environment (shown on the right) is much smaller than the variance contributed on the individual level (shown on the left). This can be interpreted as meaning the breeding patterns of the birds are quite resilient to the environmental fluctuations.}
\label{fig:rand_contribution}
\end{figure}  

\section{Conclusions}

We have presented extensions to the analysis of occupancy times for states in a compartmental model such as a stage structured demographic model from the case of constant transition rates through time previously analyzed by Roth and Caswell \cite{Roth2018} to the important case of time varying transition rates.  For demographic models, understanding how the amount of time individuals will spend in different states as environmental conditions change, either randomly or on a longer term from global change, is an important question in conservation biology, as exemplified by the demography of the California condor \cite{Meretsky2000}.  We have illustrated this with specific calculations showing how random variations in the environment affect the expected time spent as a breeder in Southern Fulmar populations.  Using these ideas in an applied context does bring up questions about the meaning of occupancy time.

We  would like to emphasise that there are a number of interpretations of occupancy times for time-dependent compartmental processes, and this has been explored only very recently in the literature \cite{Sierra_17_1,Metzler_20_1}. Firstly, it makes a difference whether one considers only a single individual organism or particle, or whether one is interested in dynamical processes that involve a lot of organisms or particles. In this article, we focused on single organisms or particles, and we used a stochastic model. When considering many organisms or particles, the law of large numbers can be used to derive a deterministic model describing the evolution of mass accurately. In our context, the deterministic evolution is given by the difference equation $x_{n+1}=B(n)x_n$ and solved using the transition operator from \eqref{determinsticevolution}, via $\Phi(n,m)v$, where $v$ is a probability vector describing the initial mass distribution. The occupancy time $\tau_R$ we have developed here carries over naturally to that setting, and the distribution of $\tau_R$ has the same interpretation, meaning that in particular Theorem~\ref{theo1} applies for $k=1$, giving the expected occupancy time. We note that higher-order moments for $k>1$ do not have a meaning in this large particle limit.

For large particle systems, a different analysis is possible for time-dependent compartmental processes, in either discrete or continuous time, that also have inputs. These are modelled in the linear case either as a nonautonomous linear difference equation $x_{n+1}=B(n)x_n + s_n$ or differential equation $\dot x = B(t)x +s(t)$. The theory developed here also applies to this setting with the interpretation that the initial mass probability distribution $v$ is determined by the input vector $s_n$ or $s(t)$ at time $n$ or $t$, respectively. This yields a (time-dependent) occupancy time describing this quantity for all particles that enter the system at a particular time, but a different perspective is possible here and was first explored in \cite{Rasmussen_16_1}. Instead of considering the future of all particles that come into the system at some time $n$ or $t$, one can also look at for how long particles leaving the system at a particular time have spent in a certain subset, which essentially covers the past of the system. We note that for autonomous systems at equilibrium, these two approaches are the same.  These two approaches  yield different answers not only for  nonautonomous systems, but also for autonomous systems that have not yet reached equilibrium.

As an example of where these different definitions of occupancy times for subsets of the states could be important in understanding dynamics, a discrete time version of the carbon cycle models considered in \cite{Rasmussen_16_1} is illustrative.  These models look at the expected time that carbon is in various pools.  The approaches we have outlined in this paper could be used for example, to determine the expected time a molecule spends in the litter (as opposed to soil or vegetation), for example, while it is in the terrestrial pool, before leaving the terrestrial pool. 

\bigskip

\textbf{Acknowledgements.}
Alan Hastings was supported by the US NSF Grant DMS-1817124. Martin Rasmussen was supported by funding from the European Union’s Horizon 2020 research and innovation programme for the ITN CRITICS under Grant Agreement Number 643073. George Chappelle's work was funded through the EPSRC CDT in Mathematics for Planet Earth.

\bibliographystyle{amsplain}
\bibliography{references}

\end{document}